\documentclass[11pt]{article}
\usepackage{amsmath,amsfonts,amscd,bezier}
\usepackage{graphicx}
\usepackage{latexsym,amssymb,amsthm}
\usepackage[utf8]{inputenc} 
\usepackage[T1]{fontenc}    
\usepackage[english]{babel}
\usepackage[usenames]{color}
\usepackage{xcolor}
\usepackage{lineno}
\usepackage[normalem]{ulem}

\newtheorem{lemma}{Lemma}[section]
\newtheorem{theorem}[lemma]{Theorem}
\newtheorem{theorem*}[]{Theorem}
\newtheorem{remark}[lemma]{Remark}
\newtheorem{proposition}[lemma]{Proposition}

\newtheorem{definition}[lemma]{Definition}
\newtheorem{example}[lemma]{Example}

\newcommand{\dd}{\textup{d}}

\begin{document}
	
\title{\bf On the Saito number of plane curves.}
	
\author{{\sc Yohann Genzmer and Marcelo E. Hernandes}\thanks{The first-named author was partially supported by the \emph{R\'eseau Franco-Brésilien en Mathématiques} (GDRI-RFBM). The second-named author was partially supported by CNPq-Brazil.}}
	
\date{}

\providecommand{\keywords}[1]
{
  \small	
  \textbf{\textit{Keywords \   --}} #1
}

\providecommand{\AMSC}[1]
{
  \small	
  \textbf{\textit{MSC 2024 --}} #1
}

\maketitle
\begin{abstract}
In this work we study the \emph{Saito number} of a plane curve and we present a method to determine the minimal Saito number for plane curves in a given equisingularity class, that gives rise to an actual algorithm. In particular situations, we also provide various formulas for this number. In addition, if $\nu_0$ and $\nu_1$ are two coprime positive integers and $N>0$ then we show that for any $1\leq k\leq \left [\frac{N\nu_0}{2}\right ]$ there exits a plane curve equisingular to the curve $$y^{N\nu_0}-x^{N\nu_1}=0$$ such that its Saito number is precisely $k$. 
\end{abstract}


\AMSC {14H20 ; 32S65, 14B05}

	
\section*{Introduction}

In his seminal article on the theory of logarithmic differential forms \cite{MR586450}, K. Saito noticed that, in dimension $2$, the module of logarithmic differential $1-$forms leaving invariant a given germ of complex plane curve $C$ is always free of rank $2$. Nowadays, known as \emph{the Saito module} of $C$ and denoted by $\Omega^1\left(\log C\right)$, it appears somehow to encode a certain amount of information about the curve itself. Of special interest are the valuations of these logarithmic differential $1-$forms and, in particular, the minimal valuation called \emph{the Saito number} of $C$, 
\[\mathfrak{s}\left(C\right)=\min_{\omega \in \Omega^1\left(\log C\right)} \nu\left(\omega \right).\]
In a series of articles, the first author studied the Saito number of a \emph{generic} curve in its equisingularity class \cite{YoyoBMS,moduligenz,genzmer2024saito,genzmer2020saito}, and as a product established various formulas or algorithms to provide the number of moduli of $C$. In \cite{MR4082253}, the two authors highlighted a link between the Tjurina number of $C$ and its Saito number, still in a generic situation. Recently, in \cite{cano2024computing}, F. Cano, N. Corral and D. Senovilla-Sanz derived from an algorithm introduced by Delorme \cite{Delorme1978,HH-algorithm}, a procedure to construct a basis of the Satio module for an irreducible curve whose has only one Puiseux pair. Moreover, in \cite{ayuso2024construction}, for a \emph{generic} irreducible curve, P. Fortuny Ayuso and J. Ribon proposed an algorithm that leads to the computation of a basis of the Saito module.
Today, the situation is ripe for further study of the non-generic case and this is the main goal of this article.

In the first section, we study the minimal Saito number in a given equisingularity class $\textup{Top}(C)$, which somehow, corresponds to the \emph{less generic} situation. We prove the following

\begin{theorem*}\label{THEO1}
    Let $C$ be a germ of complex curve in $\left(\mathbb{C}^2,0\right)$. Let $b$ be the number of punctual blowing-ups in the minimal desingularization process of $C$. Then, there exists an algorithm whose complexity is $O\left(2^b\right)$ that provides the following minimum
    \[\min_{C^\prime\in  \textup{Top}(C)}\mathfrak{s}\left(C^\prime\right)\] and the topological data associated to a $1-$form reaching the above lower bound.
\end{theorem*}

The equisingularity class $\textup{Top}(C)$ of $C$ is naturally stratified by the Saito number of its curves. In the second section, we prove that, in the one Puiseux pair case, as well as in some non-irreducible situations, the Saito numbers reach any value between its minimum and maximum value inside a given equisingularity class. As an example, we prove the following result

\begin{theorem*}\label{THEO2}
    Let $C$ be the curve given by the equation \[y^{N\nu_0}-x^{N\nu_1}=0\] where $1\leq\nu_0<\nu_1$ are coprime integers and $N\nu_0>1$. Then for any integer $k$ such that 
    \[\min_{C^\prime \in  \textup{Top}(C)}\mathfrak{s}\left(C^\prime\right)=1\leq k\leq \left\lfloor\frac{N\nu_0}{2}\right\rfloor=\max_{C^\prime \in  \textup{Top}(C)}\mathfrak{s}\left(C^\prime\right)\] there exists $C_k$ equisingular to $C$ such that 
    \[\mathfrak{s}\left(C_k\right)=k.\]
\end{theorem*}

We conjecture that the result above holds for any topological class
of plane curve.

\section{Minimal Saito number in $\textup{Top}(C)$.}  

In this section, we prove Theorem \ref{THEO1} by identifying a foliation leaving invariant a curve $C^\prime$ equisingular to $C$ and reaching the minimal Saito number.


\subsection{Dual graph of foliations.}

Let $E$ be any process of blowing-ups. The dual graph of $E$ can
be numbered by the multiplicities of its irreducible components, defined
as follows : the multiplicity $\rho\left(s\right)$ of the initial
component $s$ is $1$ and inductively, if $s$ results from the blowing-up
of a point $p\in s^{\prime}$ and $p\not\in s^{\prime\prime}$ for any other component $s^{\prime\prime}\neq s$, that is $p$ is \emph{a free point}, then $\rho\left(s\right)=\rho\left(s^{\prime}\right)$ ; if it results from the blowing-up of the point $p\in s^{\prime}\cap s^{\prime\prime}$, that is $p$ is \emph{a satellite point}, then $\rho\left(s\right)=\rho\left(s^{\prime}\right)+\rho\left(s^{\prime\prime}\right)$.
Following the definitions of \cite{hertlingfor}, for a foliation
$\mathcal{F}$, a regular curve $C$ and a point $p$ in $C$ we consider
the following indeces : the foliation being given
by the $1-$form $a\textup{d}x+b\textup{d}y\in\Omega^1$, where $C=\left\{ y=0\right\} ,$
we set
\begin{itemize}
\item if $C$ is invariant by $\mathcal{F}$, $\textup{Ind}\left(\mathcal{F},C,p\right)=\nu\left(b\left(x,0\right)\right)$
\item if $C$ is not invariant by $\mathcal{F}$,
$\textup{Tan}\left(\mathcal{F},C,p\right)=\nu\left(a\left(x,0\right)\right)$.
\end{itemize}
Notice that if $\mathcal{F}$ is singular at $p$, in any case, both
indeces above are strictly positive.\\ 
A \emph{double numbered colored graph $\mathbb{A}$ }is a graph where
the vertices are colored in black or white and are double numbered.
For any vertex $s$ of $\mathbb{A}$ and $\epsilon\in\left\{ 1,2\right\} $,
$s\left(\epsilon\right)$ stands for the $\epsilon^{\textup{th}}$
index of $s.$\\
The multiplicity of a double numbered colored graph $\mathbb{A}$
is defined by 
\[
\nu\left(\mathbb{A}\right)=-1+\sum_{s\in\mathbb{A}}s\left(1\right)s\left(2\right).
\]

\begin{definition}
\label{def:The-dual-graph}The dual graph of $\mathcal{F}$ with respect
to $E$ is the double numbered colored graph $\mathbb{A}\left[\mathcal{F},E\right]$
defined as follows:
\begin{itemize}
\item The graph of $\mathbb{A}\left[\mathcal{F},E\right]$ is the dual graph
of $E.$ 
\item If $E^{\star}\mathcal{F}$ is generically transverse to a component,
the corresponding vertex is black. If not, the vertex is white.
\item Each vertex $s$ is double numbered by 

\begin{itemize}
\item $\left(\rho\left(s\right),-\textup{val}_{w}\left(s\right)+\sum_{p\in s}\textup{Ind}\left(E^{\star}\mathcal{F},s,p\right)\right)$
if the vertex is white;
\item $\left(\rho\left(s\right),2-\textup{val}_{w}\left(s\right)+\sum_{p\in s}\textup{Tan}\left(E^{\star}\mathcal{F},s,p\right)\right)$
if the vertex is black
\end{itemize}
where $\textup{val}_{w}\left(s\right)$ stands for the number of white
vertices attached to $s.$ 
\end{itemize}
\end{definition}

In this formalism, the next property is just Theorem 3 in \cite{hertlingfor}.
\begin{proposition}[\cite{hertlingfor}]\label{Hertling}
For any foliation $\mathcal{F}$ and any blowing-up process $E$, we get
\[
\nu\left(\mathcal{F}\right)=\nu\left(\mathbb{A}\left[\mathcal{F},E\right]\right).
\]
\end{proposition}

\begin{example}\label{first.example}
Let $\mathcal{F}$ be the foliation defined by the
$1-$form
\[
\left(2x^{2}+\frac{5}{2}y^{3}-\frac{9}{2}x^{3}y\right)\textup{d}y-\left(3xy-3x^{2}y^{2}\right)\textup{d}x.
\]
It leaves invariant the curve $C=\left\{ \left(y^{2}-x^{3}\right)\left(y^{3}-x^{2}\right)=0\right\} $
and it is of multiplicity $\nu\left(\mathcal{F}\right)$ is equal to $2$. Now, if $E$ refers to the minimal
process of reduction of $C$ then the graph $\mathbb{A}\left[\mathcal{F},E\right]$
is in Figure 1. In particular, its valuation satisfies
\begin{align*}
\nu\left(\mathbb{A}\left[\mathcal{F},E\right]\right) & =-1+1\times1+2\times0+1\times1+2\times0+1\times1\\
 & =2=\nu\left(\mathcal{F}\right).
\end{align*}
\begin{center}
\begin{figure}
\begin{centering}
\includegraphics[width=6cm]{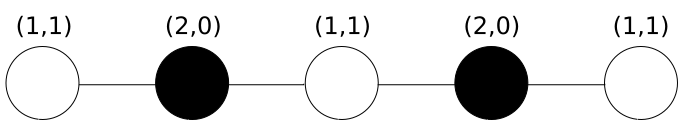}
\par\end{centering}
\caption{Dual graph of $\mathcal{F}$ with respect to the desingularization
of $C.$}
\end{figure}
\end{center}

\end{example}

\subsection{Construction of foliations with prescribed dual graphs.}
Using a method to construct germ of singular foliations in the complex plane introduced by A. Lins Neto \cite{alcides}, we can produce foliations with prescribed double numbered dual graphs. 

\begin{proposition}
\label{prop:constr}Let $\mathbb{A}$ be the dual graph of a process
of blowing-ups $E$. Consider a coloration of $\mathbb{A}$ and for any $s\in\mathbb{A}$ a double numbering $(s(1),s(2))$ such that $s(1)$ is the multiplicity $\rho\left(s\right)$. There exists a foliation $\mathcal{F}$
and a process of blowing-ups $E^{\prime}$ with the same dual graph
as $E$ such that 
\[
\mathbb{A}\left[\mathcal{F},E^{\prime}\right]=\mathbb{A}
\]
 if and only the following properties are satisfied
\begin{enumerate}
\item[a)] for any white vertex $s$, $s\left(2\right)+\textup{val}_{w}\left(s\right)\geq0;$
\item[b)] for any black vertex $s$, $s\left(2\right)+\textup{val}_{w}\left(s\right)\geq2;$
\item[c)] for any connected component $\mathbb{K}$ of $\mathring{\mathbb{A}}$
which is $\mathbb{A}$ minus the black vertices, there exists $s\in\mathbb{K}$
such that $s\left(2\right)>0.$ 
\end{enumerate}
\end{proposition}

\begin{proof}
The direct part of properties $\left(a\right)$ and $\left(b\right)$
follows from Definition \ref{def:The-dual-graph}. The direct part
of the property $\left(c\right)$ is a classical consequence of the
fact that the intersection matrix of $E$ is definite negative \cite{MR1742334}.
Indeed, if for any $s\in\mathbb{K}$, $s\left(2\right)=0$ then 
\[
\forall s\in\mathbb{K},\ \sum_{p\in s}\textup{Ind}\left(E^{\star}\mathcal{F},s,p\right)=\textup{val}_{w}\left(s\right).
\]
Since for any $s,\ s^{\prime}\in\mathbb{K}$, $s\cap s^{\prime}$
is a singular point of $E^{\star}\mathcal{F}$ for which the index
$\textup{Ind}\left(E^{\star}\mathcal{F},s,p\right)$ is bigger than
$1$, the above equalities ensures that along $\mathbb{K}$, the foliation
$E^{\star}\mathcal{F}$ is singular only at the points of type $s\cap s^{\prime}$,
$s,\ s^{\prime}\in\mathbb{K}$ with indeces equal to $1.$ The latter
is impossible according to \cite[Proposition 3.1]{MR4310291}.

Now, suppose that the properties $(a), (b)$ and $(c)$ are satisfied. The
statement relies upon a result of Lins Neto \cite{alcides} of construction
of singular foliations in dimension $2$ from elementary elements
which are reduced singularity. The only obstruction for such construction
is the well-known Camacho-Sad relation \cite{camacho}
\[
\sum_{p\in s}\textup{CS}\left(E^{\star}\mathcal{F},s,p\right)=-s\cdot s
\]
where $s\cdot s$ is the self intersection of the component $s$ and
$\textup{CS}\left(\cdot\right)$ stands for the Camacho-Sad index
of the foliation $E^{\star}\mathcal{F}$ with respect to $s$ at $p$.
More precisely, consider $s$ in $\mathbb{A}$ with $s\cdot s=-k.$
Locally around the irreducible component $s$, the total space of
$E$ is analytically equivalent to the neighborhood of $x_{1}=0$
in the following gluing model
\begin{equation}\label{Gluing}
\left(\mathbb{C}^{2},\left(x_{1},y_{1}\right)\right)\coprod\left(\mathbb{C}^{2},\left(x_{2},y_{2}\right)\right)
\end{equation}
with the identification 
\[
y_{2}=y_{1}^{k}x_{1}\quad x_{2}=\frac{1}{y_{1}}.
\]
This neighborhood can be foliated by the foliation $\mathcal{R}_{s}$
given in the coordinates $\left(x_{1},y_{1}\right)$ by the $1-$form
\[
\omega=\textup{d}x_{1}+\prod_{i=1}^{s\left(2\right)+\textup{val}_{w}\left(s\right)-2}\left(y_{1}-i\right)\textup{d}y_{1}
\]
which is completely transverse to $x_{1}=0$ except at the points
$\left(0,i\right),$ $i=1,\ldots,s\left(2\right)+\textup{val}_{w}\left(s\right)-2$
where it is tangent at order $1.$ The foliation $\mathcal{R}_{s}$
is a local model for a foliation with double numbered black graph
presented in Figure \ref{fig:Local-model-of}.
\begin{figure}
\begin{centering}
\includegraphics[width=6cm]{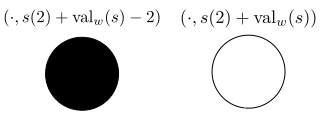}
\par\end{centering}
\caption{\label{fig:Local-model-of}Local models of elementary foliations for
black and white vertices.}
\end{figure}
The model (\ref{Gluing}) can also be foliated by $\mathcal{G}_{s}$
given by $1-$form
\begin{equation}
\eta=\frac{\textup{d}x_{1}}{x_{1}}+\sum_{i=1}^{s\left(2\right)+\textup{val}_{w}\left(s\right)}\lambda_{i}^{s}\frac{\textup{d}y_{1}}{y_{1}-i}.\label{eq:nondic}
\end{equation}
with 
\begin{equation}
\sum_{i=1}^{s\left(2\right)+\textup{val}_{w}\left(s\right)}\lambda_{i}^{s}=p.\label{eq:camacho}
\end{equation}
By construction, for any $i,$ one has 
\[
\textup{CS}\left(\mathcal{G}_{s},\left\{ x_{1}=0\right\} ,i\right)=\lambda_{i}^{s}.
\]
The foliation $\mathcal{G}_{s}$ is a local model for a foliation
with double numbered white graph presented in Figure \ref{fig:Local-model-of}. All these local models can be glued together
by gluing maps following the dual graph of $\mathbb{A}$. The property
$\left(c\right)$ ensures that, along any connected component $\mathbb{K}$
of the graph minus the black vertices, no incompatibility will occur
between the relations (\ref{eq:camacho}) and the fact that, at any
intersection point of two white components $s$ and $s^{\prime}$,
one should have
\begin{equation}
\textup{CS}\left(\mathcal{G}_{s},s,s\cap s^{\prime}\right)\cdot\textup{CS}\left(\mathcal{G}_{s^{\prime}},s^{\prime},s\cap s^{\prime}\right)=\lambda_{i_{s}}^{s}\cdot\lambda_{i_{s^{\prime}}}^{s^{\prime}}=1.\label{eq:CS}
\end{equation}
Indeed, the union of the relations (\ref{eq:camacho}) and (\ref{eq:CS})
yields a number of equations equal to 
\[
\sharp\textup{vertices}\left(\mathbb{K}\right)+\sharp\textup{edges}\left(\mathbb{K}\right).
\]
However, the number of variables involved in the mentioned relations
is 
\[
\sum_{s\in\mathbb{K}}s\left(2\right)+\textup{val}_{w}\left(s\right).
\]
Following property $\left(c\right)$ the above number of variables
satisfies
\[
\begin{array}{ll}
\sum_{s\in\mathbb{K}}s\left(2\right)+\textup{val}_{w}\left(s\right)\geq1+\sum_{s\in\mathbb{K}}\textup{val}_{w}\left(s\right) & =2\cdot\sharp\textup{vertices}\left(\mathbb{K}\right)-1 \\ & =\sharp\textup{vertices}\left(\mathbb{K}\right)+\sharp\textup{edges}\left(\mathbb{K}\right)
\end{array}\]
since $\sharp\textup{vertices}\left(\mathbb{K}\right)=\sharp\textup{edges}\left(\mathbb{K}\right)+1.$
Therefore the system of equations (\ref{eq:camacho}) and (\ref{eq:CS})
has always a solution - that can be chosen rational.

As a whole, the gluings lead to a foliation defined in a neighborhood of a divisor $\mathcal{D}$
with same intersection matrix as the one of the exceptional divisor
of $E.$ According to a classical result of H. Grauert \cite{grauerthans},
the neighborhood of $\mathcal{D}$ is analytically equivalent to the
neighborhood of the exceptional divisor of some blowing-up process
$E^{\prime}$ with same dual graph as $E.$ The latter neighborhood
is foliated by a foliation $\mathcal{F}^{\prime}$ that can be contracted
by $E^{\prime}$ in a foliation $\mathcal{F}.$ By construction, one
has 
\[
\mathbb{A}\left[\mathcal{F},E^{\prime}\right]=\mathbb{A}.
\]
\end{proof}

\subsection{Minimal Saito number of a given topological class of curve.}

Let $C$ be a germ of curve in the complex plane defined by $f\in\mathbb{C}\{x,y\}$. The \emph{Saito module} of $C$ is the $\mathbb{C}\{x,y\}$-module $$\Omega^1\left(\log C\right)=\{\omega\in\Omega^1:\ \omega\wedge\textup{d}f\in \langle f\rangle\}.$$
Saito, in \cite{MR586450}, showed that, in this context, $\Omega^1\left(\log C\right)$ is a free-module generated by a set of two elements $\{\omega_1,\omega_2\}$, called a \emph{Saito basis} for $C$, and it is characterized by a property we refer to as the \emph{Saito's criterion} : a set of two elements $\{\omega_1,\omega_2\}$ is a Saito basis if and only if for some unit $u\in\mathbb{C}\{x,y\}$, we get $$\omega_1\wedge\omega_2=uf\textup{d}x\wedge\textup{d}y.$$ The \emph{Saito number} of $C$ is by definition $$\mathfrak{s}(C)=\min\{\nu(\omega_1),\nu(\omega_2)\}.$$ It can be seen that $\mathfrak{s}(C)$ is also equal to $\min\{\nu(\mathcal{F}):\ C\ \mbox{is invariant by}\ \mathcal{F}\}$ and is an analytic invariant of $C$ \cite{YoyoBMS}. 

Let $\textup{Top}(C)$ be the topological (or equisingularity) class of $C$. In this section, as a consequence of the previous results, we present an algorithm to select a curve $C^{\prime}\in \textup{Top}(C)$ and a foliation $\mathcal{F}^{\prime}$ leaving invariant $C^{\prime}$ such that 
\[\nu(\mathcal{F}^{\prime})=\mathfrak{s}(C^{\prime})=\min_{C^{\prime\prime}\in \textup{Top}(C)}\mathfrak{s}(C^{\prime\prime}).\] 

Let $E$ be the
desingularization process of $C$ and $\mathbb{A}$ be the dual graph
of $E.$ The integer $n_{s}\left(C\right)$ stands for the number
of components of the strict transform of $C$ by $E$ attached to
the component $s$ of the exceptional divisor. 

Let $\mathcal{F}$ be a foliation whose $C$ is an invariant curve
and consider its dual graph $\mathbb{A}\left[\mathcal{F},E\right].$
Suppose that $s$ is a white vertex. Since $s$ is invariant by $E^{\star}\mathcal{F}$,
$E^{\star}\mathcal{F}$ is singular at any point $p$ of $s$ at which
is attached a component of the strict transform of $C.$ In particular,
the index 
$$\textup{Ind}\left(E^\star\mathcal{F},s,p\right)$$ is strictly positive. Therefore, for any white vertex $s,$
one has 
\[
s\left(2\right)\geq n_{s}\left(C\right).
\]
From $\mathbb{A}\left[\mathcal{F},E\right]$, we consider the double
numbered colored graph $\mathbb{A}$ which is a copy of $\mathbb{A}\left[\mathcal{F},E\right]$
except that we set 
\begin{itemize}
\item for any white vertex $s\left(2\right)=n_{s}\left(C\right)$
\item for any black vertex $s\left(2\right)=2-\textup{val}_{w}\left(s\right).$
\end{itemize}
If doing so, the property $\left(c\right)$ of Proposition \ref{prop:constr}
is not satisfied for some component $\mathbb{K}$, we set $s\left(2\right)=1$
for the vertex $s$ of $\mathbb{K}$ for which $\rho\left(s\right)$
is minimal. Applying Proposition \ref{prop:constr} yields a foliation
$\mathcal{F}^{\prime}$ that let invariant a curve $C^{\prime}$ which
is topologically equivalent to $C$. Indeed, for any component $s$,
either $s$ is black, $\mathcal{F}^{\prime}$ generically transverse
to $s$ and we choose arbitrarily $n_{S}\left(C\right)$ regular and
transverse invariant curves attached to $s$, or $s$ is white and $\mathcal{F}^{\prime}$
is locally given by (\ref{eq:nondic}) and leaves invariant $s\left(2\right)=n_{s}\left(C\right)$
regular and transverse curves still attached to $s.$ The union of
all these curves yields a curve $C^{\prime}$ whose desingularization
process has the dual graph of $E$. Since $n_{s}\left(C^{\prime}\right)=n_{s}\left(C\right)$,
then $C^{\prime}$ and $C$ are topologically equivalent \cite{zariskitop}. 
\begin{proposition}[Algorithm 1]
\label{prop:algo} We have
\[
\nu\left(\mathcal{F}^{\prime}\right)\leq\nu\left(\mathcal{F}\right).
\]
\end{proposition}

\begin{proof}
Let $\mathcal{K}_{1}$ the set of components $\mathbb{K}$ of $\mathring{\mathbb{A}}$
for which one has $n_{s}\left(C\right)=0$ for all $s\in\mathbb{K}$
and $\mathcal{K}_{2}$ the other set of components. If $\mathbb{K}\in\mathcal{K}_{1},$
we denote by $s_{\mathbb{K}}$ a vertex for which $s_{\mathbb{K}}\left(2\right)>0$,
that exists according to Proposition \ref{prop:constr}.

We have
\begin{align*}
\nu\left(\mathcal{F}\right) & =-1+\sum_{s\in\mathbb{A}\left[\mathcal{F},E\right]}s\left(1\right)s\left(2\right)\\
 & =-1+\sum_{s\textup{ black}}\rho\left(s\right)\left(2-\textup{val}_{w}\left(s\right)+\sum_{p\in s}\textup{Tan}\left(E^{\star}\mathcal{F},s,p\right)\right)\\
 &\hspace{1cm} +\sum_{\mathbb{K}\in\mathcal{K}_{1}}\sum_{s\in\mathbb{K}}\rho\left(s\right)s\left(2\right)+\sum_{\mathbb{K}\in\mathcal{K}_{2}}\sum_{s\in\mathbb{K}}\rho\left(s\right)s\left(2\right).
\end{align*}
Now, for $\mathbb{K}\in\mathcal{K}_{1}$, we can give the following
lower bound 
\begin{align*}
\sum_{s\in\mathbb{K}}\rho\left(s\right)s\left(2\right) & \geq\rho\left(s_{\mathbb{K}}\right)s_{\mathbb{K}}\left(2\right)+\sum_{s\in\mathbb{K}\setminus\left\{ s_{\mathbb{K}}\right\} }\rho\left(s\right)s\left(2\right)\\
 & \geq\min_{s\in\mathbb{K}}\rho\left(s\right).
\end{align*}
For $\mathbb{K}\in\mathcal{K}_{2},$ the following lower bound occurs
\[
\sum_{s\in\mathbb{K}}\rho\left(s\right)s\left(2\right)\geq\sum_{s\in\mathbb{K}}\rho\left(s\right)n_{s}\left(C\right).
\]
Finally, if $s$ is black, then 
\[
\rho\left(s\right)\left(2-\textup{val}_{w}\left(s\right)+\sum_{p\in s}\textup{Tan}\left(E^{\star}\mathcal{F},s,p\right)\right)\geq\rho\left(s\right)\left(2-\textup{val}_{w}\left(s\right)\right).
\]
Combining all these inequalities leads to 
\begin{align*}
\nu\left(\mathcal{F}\right) & \geq-1+\sum_{s\textup{ black}}\rho\left(s\right)\left(2-\textup{val}_{w}\left(s\right)\right)\\
 & \hspace{1cm}+\sum_{\mathbb{K}\in\mathcal{K}_{1}}\min_{s\in\mathbb{K}}\rho\left(S\right)+\sum_{\mathbb{K}\in\mathcal{K}_{2}}\sum_{s\in\mathbb{K}}\rho\left(s\right)n_{s}\left(C\right)=\nu\left(\mathcal{F^{\prime}}\right).
\end{align*}
\end{proof}
Propostion \ref{prop:algo} provides a simple algorithm that determines the minimal Saito number of a given equisingularity class of plane curves. Consider $E$ the desingularization process
of $C$ and $\mathbb{A}$ its dual graph. Choose any coloration of
the graph among the finite set of such coloration and set $s\left(2\right)$
as in Proposition \ref{prop:algo} : the value of $s\left(2\right)$
depends only on $C$ and on the chosen coloration. Each such choice
leads to a certain multiplicity of foliation. The smallest
multiplicity among them is the desired number. The complexity of this algorithm is $O(2^b)$ where $b$ is the length of the desingularization process. As a consequence, Theorem \ref{THEO1} stated in the introduction is proved.

In the sections below, we treat some examples, for which, beyond the algorithm, some formula can be established. 

\subsubsection{Irreducible curve.}
Let $C$ be an irreducible plane curve given by a parametrization $$\psi(t)=\left(t^{\nu_0},\sum_{i\geq \beta_1}a_it^{i}\right).$$ The equisingularity class $\textup{Top}(C)$ of $C$ can be totally determined by the dual graph of the desingularization process $E$ of $C$, equivalently by its characteristic exponents $\beta_0=\nu_0=\nu(C), \beta_1=\nu_1,\beta_2,\ldots ,\beta_g$ or by its value semigroup $\Gamma=\langle \nu_0,\nu_1,\ldots ,\nu_g\rangle$ (see \cite{zariski} for instance).

Notice that $C$ can be defined by the minimal polynomial $f\in\mathbb{C}\{x\}[y]$ of $$\sum_{i\geq \beta_1}a_ix^{\frac{i}{\nu_0}}.$$ Let $f_g\in\mathbb{C}\{x\}[y]$ be the minimal polynomial of the function $$\sum_{\beta_1\leq i< \beta_g}a_ix^{\frac{i}{\nu_0}},$$ according to \cite{zariski} we get $\nu(f_g)=\frac{\nu_0}{e_{g-1}}$ where $$e_{g-1}=gcd(\beta_0,\ldots ,\beta_{g-1})=\gcd(\nu_0,\ldots ,\nu_{g-1})$$ and the intersection multiplicity with $f$ is $\textup{I}(f,f_g)=\nu_g$.\\ Consider the differential $1-$form $\omega=\nu_0x\textup{d}f_g-\nu_gf_g\textup{d}x$. The foliation associated to $\omega$ leaves invariant the curves $x=0$ and $f_g=0$ whose strict transform are attached respectively in the first and last \emph{extreme}\footnote{{A extreme component of $E$ is an irreducible component that intersects only one other irreducible component of $E$.}} \emph{component} of $E$ (see Theorem 3.7 in \cite{CHH}). In addition, it is purely dicritical along the \emph{central component}, that is the one to which is attached the strict transform of $C$. In such component, the foliation admits infinite many invariant curves $C_u$ given by $$\psi_u(t)=\left (t^{\nu_0},\sum_{\beta_1\leq i<\beta_g}a_it^{i}+ut^{\beta_g}+\sum_{i>\beta_g}r(u)t^i\right )$$ for every $u\in\mathbb{C}\setminus\{0\}$ and $r(u)\in\mathbb{C}(u)$. Consequently, $C_u\in \textup{Top}(C)$ for any $u\in\mathbb{C}\setminus\{0\}$. In addition, we get 
\[\nu(\omega)=\frac{\nu_0}{e_{g-1}}.\]

On another hand, the simple structure of $E$ allows us to
follow Algorithm 1 described in Proposition \ref{prop:algo} \emph{by hand.} Indeed, the desingularization
process of $C$ is shown in Figure \ref{fig:Red-Irre}.
\begin{figure}
\begin{centering}
\includegraphics[width=6cm]{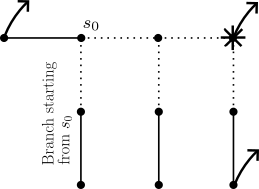}
\par\end{centering}
\caption{\label{fig:Red-Irre}Desingularization process for a irreducible plane curve and $\mathcal{F}_{\textup{min}}$.}
\end{figure}
Consider a foliation $\mathcal{F}_{\textup{min}}$ purely dicritical along the central component. Moreover, $\mathcal{F}_{\textup{min}}$ leaves invariant two regular curve attached to the first and the last extreme component of $E$ of respective multiplicities $1$  and $\frac{\nu(C)}{e_{g-1}}$. Applying Proposition \ref{Hertling} in that situation yields the formula
\[\nu\left(\mathcal{F}_{\textup{min}}\right)=\frac{\nu_0}{e_{g-1}}.\]
This value is also the minimum value for a Saito number in the equisingularity class of $C$. Indeed, consider a foliation $\mathcal{F}$  tangent to $C$. Proposition \ref{Hertling} is written
\[\nu\left(\mathcal{F}\right)=-1+\sum_{s\in\mathbb{A}[\mathcal{F},E]}s(1)s(2).\]
In view of the expression of $s(2)$, a term with a negative contribution in the above sum may appear only if $\mathcal{F}$ is dicritical along some component $s_0$ of valence $3$ and any component attached to $s_0$ is not dicritical. However, doing so, along the branch of the tree attached to $s_0$ there must be some component $s_1$ such that $s_1(2)>0$. Therefore, the contribution of both components $s_0$ and $s_1$ is written 
\[s_0(1)(-1+\alpha)+s_1(1)s_1(2)\] with $\alpha\geq 0$. Since, $s_1(1)\geq s_0(1)$, as the whole, the contribution keeps on being positive : as a consequence, it is \emph{useless} for $\mathcal{F}$ to be dicritical along $s_0$ in order to reach the desired minimum and the valuation of $\mathcal{F}$ is bigger than the valuation of $\mathcal{F}_{\textup{min}}$. Finally, we recover the result
of \cite{MR4310291}, that is  
\[
\min_{C^{\prime}\in\textup{Top}\left(C\right)}\mathfrak{s}\left(C^{\prime}\right)=\nu(\omega)=\nu\left(\mathcal{F}_{\textup{min}}\right)=\frac{\nu_0}{e_{g-1}}.
\]

\begin{remark} The Saito numbers of the curves in $\textup{Top}(C)$ where $C$ is an irreducible curve for which $e_{g-1}=2$, are constant. Indeed, from the result above and \cite{YoyoBMS}, we have 
\[
\min_{C^{\prime}\in\textup{Top}\left(C\right)}\mathfrak{s}\left(C^{\prime}\right)=\max_{C^{\prime}\in\textup{Top}\left(C\right)}\mathfrak{s}\left(C^{\prime}\right)=\frac{\nu_0}{2}.\]
\end{remark}

\subsubsection{Non-irreducible Curves.}

\paragraph{A curve with two components.}
As a generalization of Example \ref{first.example} let us consider $C$ be the plane curve defined by \[f=\left(y^{\nu_0}-x^{\nu_1}\right)\left(x^{\nu_0}-y^{\nu_1}\right)=0\] where $1=gcd(\nu_0,\nu_1)<\nu_0<\nu_1$. 
It can be seen that the two following differential $1-$forms 
\begin{eqnarray*}
\omega_1&=&\left (
\nu_1^2x^{\nu_1-\nu_0}(y^{\nu_1}-x^{\nu_0})-\nu_0^2(y^{\nu_0}-x^{\nu_1})\right )dx\\
&&+\nu_0\nu_1xy^{\nu_0-1}(1-x^{\nu_1-\nu_0}y^{\nu_1-\nu_0})dy\\
\omega_2&=&\left (\nu_1^2y^{\nu_1-\nu_0}(y^{\nu_0}-x^{\nu_1})-
\nu_0^2(y^{\nu_1}-x^{\nu_0})\right )dy\\
&&-\nu_0\nu_1x^{\nu_0-1}y(1-x^{\nu_1-\nu_0}y^{\nu_1-\nu_0})dx    
\end{eqnarray*}
satisfy the Saito criterion, that is $$\omega_1\wedge \omega_2=(\nu_1-\nu_0)(\nu_1+\nu_0)(\nu_0^2-\nu_1^2x^{\nu_1-\nu_0}y^{\nu_1-\nu_0})f\dd x\wedge \dd y$$ and thus  $\{\omega_1,\omega_2\}$ is a Saito basis for $C$. As a consequence, we find \[\mathfrak{s}(C)=\nu_0.\] 
Since the latter is also an upper bound for the maximum Saito number in the associated equisingularity class, we obtain 
{\[\max_{C'\in\textup{Top}(C)}\mathfrak{s}(C')=\nu_0.\]}
On the other hand, the minimal Saito number is reached for the foliation depicted in Figure \ref{fig:Min.C}, and thus 
{\[\min_{C'\in\textup{Top}(C)}\mathfrak{s}(C')=2.\]}
\begin{figure}
\begin{centering}
\includegraphics[width=4.5cm]{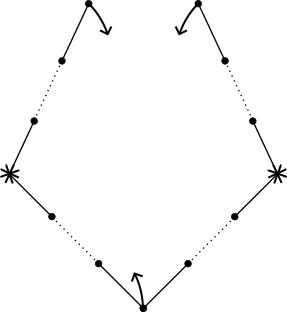}
\par\end{centering}
\caption{\label{fig:Min.C}Topology of the foliation with the minimal Saito number for $C:(y^{\nu_0}-x^{\nu_1})(x^{\nu_0}-y^{\nu_1})=0$.}
\end{figure}

\paragraph{Curves with many components.}
Let $C$ be a curve desingularized by $E$. Let $\mathbb{A}$ be the dual graph of $E$. We say that $C$ \emph{has a lot of components} if
\[\forall s\in\mathbb{A},\  n_s(C)\geq 2+\sum_{s^\prime,s\cap s^\prime\neq\emptyset}\frac{\rho(s^\prime)}{\rho(s)}. \]
The curves with a lot of components are of special interest because the foliation associated to their minimal Saito number happens to be \emph{absolutely dicritical} as defined in \cite{MR2806526}. 
\begin{proposition}
If $C$ has a lot of components 
then the minimal Saito number in $\textup{Top}(C)$ is equal to
\[\min_{C^{\prime}\in\textup{Top}(C)}\mathfrak{s}(C^{\prime}) =-1+2\sum_{s\in \mathbb{A}}\rho(s).\]
and is reached by an absolutely dicritical foliation with respect to $E$.
\end{proposition}
In particular, the minimal Saito number of a topological class of curve with a lot of components does not depends on $C$ anymore but only on $E$.
\begin{proof}
Let $\mathcal{F}$ be a foliation leaving invariant $C^{\prime}$ and reaching the minimum above. According to Proposition \ref{prop:algo}, we can suppose that $\mathcal{F}$ is constructed by gluing local models of Proposition \ref{prop:constr}. Consider now $s\in \mathbb{A}[\mathcal{F},E]$ and suppose that $s$ is white. Using still the Lins Neto's argument, we can construct a foliation $\mathcal{F}_s$ such that $\mathbb{A}[\mathcal{F}_s,E]$ has the same coloration as $\mathbb{A}[\mathcal{F}_s,E]$ except that $s$ is black in $\mathbb{A}[\mathcal{F}_s,E]$. In particular, following Proposition \ref{Hertling}
\begin{eqnarray*}
\nu(\mathcal{F})&=&-1+\rho(s)n_s(C)
\\ & & +\sum_{s^\prime,s\cap s^\prime\neq\emptyset}\delta_{s^\prime} \rho(s^\prime)n_{s^\prime}(C)+(1-\delta_{s^\prime})\rho(s^\prime)(2-\textup{val}_w(s^\prime))\\
& & +(\textup{does not depend on $s$})
\end{eqnarray*}
and 
\begin{eqnarray*}
\nu(\mathcal{F}_s)&=&-1+\rho(s)(2-\textup{val}_w(s))
\\& & +\sum_{s^\prime,s\cap s^\prime\neq\emptyset}\delta_{s^\prime} \rho(s^\prime)n_{s^\prime}(C)+(1-\delta_{s^\prime})\rho(s^\prime)(3-\textup{val}_w(s^\prime))
\\& &+(\textup{does not depend on $s$})
\end{eqnarray*} 
where $\delta_s=1$ if $s$ is white, $0$ otherwise. Therefore, we get 
\[\nu(\mathcal{F})-\nu(\mathcal{F}_s)=\rho(s)(n_s(C)-2+\textup{val}_w(s))-\sum_{s^\prime,s\cap s^\prime\neq\emptyset}(1-\delta_{s^\prime})\rho(s^\prime).\]
It can be seen that the latter is a positive expression under the assumption of the proposition. As a consequence, we can always decrease the multiplicity of a foliation leaving invariant a curve $C^{\prime}\in\textup{Top}(C)$ by making black any of component $\mathbb{A}[\mathcal{F},E]$. In the end, the obtained foliation is \emph{absolutely dicritical} as defined in \cite{MR2806526} with respect to $E$ and its multiplicity is \[-1+\sum_{s\in \mathbb{A}} \rho(s)(2-\textup{val}_w(s))=-1+2\sum_{s\in \mathbb{A}} \rho(s).\]
\end{proof}

\section{Range of the Saito function on given topological classes.}

Let us consider $C$ a germ of plane curve. In  \cite{genzmer2020saito}, it is proved that the maximum Saito number along the topological class $\textup{Top}\left(C\right)$ is given by
\[
\max_{C^{\prime}\in\textup{Top}\left(C\right)}\mathfrak{s}(C^{\prime})=\left\{ \begin{array}{rl}
\frac{\nu\left(C\right)}{2}-1 & \textup{ if \ensuremath{C} is \emph{radial} and \ensuremath{\nu\left(C\right)} is even}\vspace{0.2cm}\\
\left\lfloor \frac{\nu\left(C\right)}{2}\right\rfloor  & \textup{ if not,}
\end{array}\right.
\]
\emph{radial} being defined in \cite{genzmer2020saito}. Actually, the above maximum is reached for a generic element $C^{\prime}$ in the equisingularity class of $C$. Moreover,
in the previous section, we provide an algorithm {(see Proposition \ref{prop:algo})} to compute the minimum
$
\min_{C^{\prime}\in\textup{Top}\left(C\right)}\mathfrak{s}(C^{\prime}).$ It is of natural interest to look at the integers between these two bounds that are reached as a Saito number of a certain analytical class of curves in the given equisingularity class. 

For that purpose, let us consider the curve $C_{N,\nu_{0},\nu_{1}}$ given by 
\[
f_{N,\nu_{0},\nu_{1}}=y^{N\nu_{0}}-x^{N\nu_{1}}=0
\]
where $N>0$ and $\nu_{0}\leq\nu_{1}$ are relatively prime. Since, considering
the two $1-$forms 
\[
\omega=\nu_{1}y\textup{d}x-\nu_{0}x\textup{d}y\textup{ and }\eta=\textup{d}f_{N,\nu_{0},\nu_{1}}
\]
leads to the Saito criterion 
\[
\omega\wedge\eta=N\nu_{0}\nu_{1}f_{N,\nu_{0},\nu_{1}}\textup{d}x\wedge\textup{d}y,
\]
we obtain
\[
\min_{C\in\textup{Top}\left(C_{N,\nu_{0},\nu_{1}}\right)}\mathfrak{s}(C)=\min\{1,N\nu_0-1\},
\] which is equal to $1$ except when $N\nu_0=1$, that is when the curve $C_{N,\nu_{0},\nu_{1}}$ is regular, for which the minimum in $0$. As a consequence of the computations in \cite[Proposition 8]{YoyoBMS}, it appears that the curve $C_{N,\nu_{0},\nu_{1}}$ is radial if and only if $\nu_{1}=1$
and $N\geq 3.$
In particular, if $\nu_1>1$  then \[\max_{C\in\textup{Top}\left(C_{N,\nu_{0},\nu_{1}}\right)}\mathfrak{s}(C)=\left\lfloor\frac{N\nu_0}{2}\right\rfloor.\]
The goal of the next subsections is to show that the range of the map
\[C\in \textup{Top}\left(C_{N,\nu_0,\nu_1}\right)\mapsto \mathfrak{s}(C)\]
is the whole set of integers between the two above extrema.

\subsection{Saito numbers in the topological class $\textup{Top}(C_{1,\nu_0,\nu_1})$.}

If $\nu_0=1$ then $\mathfrak{s}(C)=0$, thus, in what follows, we suppose that $\nu_0>1$. In particular, we obtain 
\[\min_{C\in\textup{Top}\left(C_{1,\nu_{0},\nu_{1}}\right)}\mathfrak{s}(C)=1\quad \textup{ and } \quad\max_{C\in\textup{Top}\left(C_{N,\nu_{0},\nu_{1}}\right)}\mathfrak{s}(C)=\left\lfloor\frac{\nu_0}{2}\right\rfloor.\]

Since $\mathfrak{s}(C_{1,\nu_0,\nu_1})=1$, to show that any $k\in\{1,\ldots ,\left \lfloor \frac{\nu_0}{2}\right \rfloor\}$ is achieved as a Saito number for an element in $\textup{Top}(C_{1,\nu_0,\nu_1})$ it is sufficient to consider $k>1$ and $\nu_0>3$.

Given any irreducible plane curve $C$ with parameterization $\psi(t)\in\mathbb{C}\{t\}\times\mathbb{C}\{t\}$ the semigroup $\Gamma_C$ associated to $C$ is 
$$\Gamma_C=\{\nu(\psi^*(h));\ h\in\mathbb{C}\{x,y\}\ \mbox{such that}\ \psi^*(h)\neq 0\}.$$ We can extend the valuation $\nu(\cdot)$ to a differential $1$-form not tangent to $C$ and we define the set
$$\Lambda_C=\{\nu(\psi^*(\eta))+1;\ \eta\ \mbox{is a differential $1$-form not tangent to}\ C \}.$$
The set $\Lambda_C$
is an analytical invariant for $C$ and it is a $\Gamma_C$-semimodule finitely generated \cite{HH-algorithm}, that is,
there exist 
$\lambda_{-1},\lambda_{0},\ldots ,\lambda_s\in\Lambda_C$ such that $$\Lambda_C=\bigcup_{i=-1}^s(\Gamma_C+\lambda_i)$$ with $\lambda_j\not\in\Lambda_{j-1}:=\bigcup_{i=-1}^{j-1}(\Gamma_C+\lambda_i)$ for $0\leq j\leq s$. In \cite{HH-algorithm} we find an algorithm to compute a minimal system of generators of {$\Lambda_C$} for any irreducible plane curve $C$. 

Given $C\in\textup{Top}\left(C_{1,\nu_0,\nu_1}\right)$, that is, $\Gamma_C=\langle \nu_0,\nu_1\rangle$ there is another valuation associated to $C$ called \emph{divisorial valuation} given as the following. If $h=\sum_{i,j\geq 0}h_{ij}x^iy^j\in\mathbb{C}\left\{ x,y\right\}\setminus\{0\}$ the divisorial valuation $\nu_{D}(h)$ of $h$ is (see \cite{cano2024computing})
\begin{equation}\label{divisorial-function}
\nu_{D}\left(h\right)=\min\left \{\left.\nu_{0}i+\nu_{1}j\right| \ h_{ij}\neq0\right \}
\end{equation}
and we extend it to a differential $1$-form $A\textup{d}x+B\textup{d}y$ by $$\nu_{D}\left(A\textup{d}x+B\textup{d}y\right)=\min\{\nu_{D}(A)+\nu_0, \nu_{D}(B)+\nu_1\}.$$

For $C\in\textup{Top}(C_{1,\nu_0,\nu_1})$, Cano, Corral and Senovilla-Sanz in  \cite{cano2024computing} introduce a finite set of integers from which is derived a characterization of a Saito basis for $C$. In the following, we briefly describe this construction. Let $\lambda_{-1},\lambda_{0},\ldots ,\lambda_s\in\Lambda_C$ be the miminal generators for $\Lambda_C$. Setting $t_0=\lambda_{0}=\nu_1$, for $1\leq i\leq s+1$, define inductively the following data
\begin{equation}\label{integers}
\begin{array}{lll}
u_i^{\star}&=&\min\{\lambda_{i-1}+\nu_\star n\in\Lambda_{i-2};\ n\geq 1\}\\
t_i^{\star}&=&t_{i-1}+u_i^{\star}-\lambda_{i-1}
\end{array}
\end{equation}
where $\star=0,1$ and
\[t_i=\min\{t_i^{0},t_i^{1}\} \quad \tilde{t_i}=\max\{t_i^{0},t_i^{1}\}.\]
The mentioned above result is enunciated below.
\begin{theorem}[Cano, Corral and Senovilla-Sanz] For $C\in\textup{Top}(C_{1,\nu_0,\nu_1})$, there exist two differential $1$-forms $\omega_{s+1}$ and $\tilde{\omega}_{s+1}$ leaving $C$ invariant such that
	$$\nu_{D}(\omega_{s+1})=t_{s+1}\ \ \mbox{and}\ \ \nu_{D}(\tilde{\omega}_{s+1})=\tilde{t}_{s+1}.$$
Moreover, for any pair of differential $1$-forms as above, the set $\{\omega_{s+1},\tilde{\omega}_{s+1}\}$ is a Saito basis for $C$.
\end{theorem}\label{theo-CanoCorral}

For $\nu_0>3$ and any $1<k\leq \left\lfloor \frac{\nu_0}{2}\right \rfloor$, let us consider the differential $1$-form defined by
\begin{equation}\label{1form}
\omega=\nu_1x^{k-1}(\nu_0xdy-\nu_1ydx)+\nu_0(\gamma -\nu_1)y^{\nu_0-k}dy
\end{equation} 
with $\gamma:=(\nu_0-k+1)\nu_1-k\nu_0$. 

Notice that for any plane curve $C\in\textup{Top}\left(C_{1,\nu_0,\nu_1}\right)$ we get
\[
\nu_{D}\left(\omega\right)=\nu_{D}\left(x^{k-1}\left(\nu_0x\textup{d}y-\nu_1y\textup{d}x\right)\right)=k\nu_0+\nu_1.
\]

\begin{lemma}\label{lemma-aux}
There is a curve $C\in\textup{Top}(C_{1,\nu_0,\nu_1})$ invariant by $\omega$ and given by a parametrization of the form
\begin{equation}\label{parametrization}\psi(t)=\left ( t^{\nu_0},t^{\nu_1}+t^{\gamma}+\sum_{i\geq 2\gamma-\nu_1}a_it^i\right ).\end{equation}   
\end{lemma}
\begin{proof}
Considering $\psi(t)$ as in the lemma, we have
$$\psi^*(\nu_0xdy-\nu_1ydx)=\nu_0(\gamma -\nu_1)t^{\gamma+\nu_0-1}+\sum_{i\geq 2\gamma-\nu_1}\nu_0(i-\nu_1)a_it^{i+\nu_0-1}$$
and $$\psi^*(y^{\nu_0-k+1})=t^{(\nu_0-k+1)\nu_1}+(\nu_0-k+1)t^{\gamma+(\nu_0-k)\nu_1}+\sum_{j>(\nu_0-k)\nu_1+\gamma}Q_jt^{j}$$ where $Q_j\in\mathbb{C}[a_{2\gamma-\nu_1+1},\ldots ,a_{j-(\nu_0-k)\nu_1}]$. In this way, since $$\gamma=(\nu_0-k+1)\nu_1-k\nu_0$$ we get 
\begin{eqnarray*}\psi^*(y^{\nu_0-k}dy)&=&\nu_1t^{(\nu_0-k+1)\nu_1-1}+(\gamma+(\nu_0-k)\nu_1)t^{\gamma+(\nu_0-k)\nu_1-1} \\ &&+\sum_{j>(\nu_0-k)\nu_1+\gamma}\frac{j}{\nu_0-k+1}Q_jt^{j-1}  \\
&=& \nu_1t^{\gamma+k\nu_0-1}+(\gamma+(\nu_0-k)\nu_1)t^{2\gamma-\nu_1+k\nu_0-1}+ \\ &&+\sum_{i>2\gamma-\nu_1}\frac{i+k\nu_0}{\nu_0-k+1}Q_{i+k\nu_0}t^{i+k\nu_0-1}
\end{eqnarray*}
and
\begin{eqnarray*}
\psi^*(\omega)&=&\nu_0(\gamma-\nu_1)\left (2\nu_1a_{2(\gamma-\nu_1)} -(\gamma+(\nu_0-k)\nu_1 )\right )t^{\gamma+(\nu_0-k)\nu_1-1}\\
&&+\nu_0\sum_{i> 2\gamma-\nu_1}\left (\nu_1(i-\nu_1)a_i-\frac{(i+k\nu_0)(\gamma-\nu_1)}{\nu_0-k+1}Q_{i+k\nu_0} \right )t^{i+k\nu_0-1}.
\end{eqnarray*}

From $k\leq \nu_0-k$, we get $i+k\nu_0-(\nu_0-k)\nu_1<i$ and $$Q_{i+k\nu_0}\in\mathbb{C}[a_{2\gamma-\nu_1},\ldots ,a_{i+k\nu_0-(\nu_0-k)\nu_1}]\subseteq\mathbb{C}[a_{2\gamma-\nu_1},\ldots ,a_{i-1}].$$ So, setting
$$a_{2\gamma-\nu_1}=\frac{\gamma+(\nu_0-k)\nu_1}{2\nu_1}\ \ \ \mbox{and}\ \ \ a_i=\frac{(\gamma-\nu_1)(i+k\nu_0)}{\nu_1(i-\nu_1)(\nu_0-k+1)}Q_{i+k\nu_0}$$
yields a parameterization $\psi(t)$ defining a plane curve $C\in\textup{Top}(C_{1,\nu_0,\nu_1})$ invariant by $\omega$.
\end{proof}
We now prove the main result of this subsection.
\begin{proposition}\label{irred}
	For any $1\leq k\leq \left \lfloor \frac{\nu_0}{2}\right \rfloor$,
	there exists $C\in\textup{Top}\left(C_{1,\nu_0,\nu_1}\right)$ such that 
	$\mathfrak{s}(C)=k$.
\end{proposition}
\begin{proof}
If $k=1$ then considering $C_{1,\nu_0,\nu_1}$ we get $\mathfrak{s}(C_{1,\nu_0,\nu_1})=1$. Therefore, suppose that $\nu_0>3$ and let
$k$ be an integer in $\left\{ 2,\ldots,\left \lfloor \frac{\nu_0}{2}\right \rfloor\right\} $.
Let $C\in\textup{Top}(C_1,\nu_0,\nu_1)$ be the curve characterized in the previous lemma with parametrization $\psi(t)$ as (\ref{parametrization}). 

We will compute the minimal generators $\{\lambda_{-1},\lambda_0,\ldots ,\lambda_s\}$ for $\Lambda_C$ using the algorithm developed in \cite{HH-algorithm}. The first two generators are written 
\begin{align*}
\lambda_{-1}&=\nu(\psi^*(\textup{d}x))=\nu_0,\\
\lambda_{0}&=\nu(\psi^*(\textup{d}y))=\nu_1.    
\end{align*} The next generator of $\Lambda_C$ will be equal to the valuation $\nu(\psi^*(\omega_1^{(i)}))$ where $$\nu(\psi^*(\omega_1^{(i)}))\not\in(\Gamma_C+\nu_0)\cup (\Gamma_C+\nu_1)\subset\Gamma_C$$ for $i\in\{1,2\}$ and
\begin{align*}
\omega_1^{(1)}&=\nu_0x\textup{d}y-\nu_1y\textup{d}x+h_{11}\textup{d}x+h_{12}\textup{d}y \\ \omega_1^{(2)}&=\nu_0y^{\nu_0-1}\textup{d}y-\nu_1x^{\nu_1-1}\textup{d}x+h_{21}\textup{d}x+h_{22}\textup{d}y.
\end{align*}
The function $h_{ij}\in\mathbb{C}\{x,y\}$ will satisfy furthermore
\begin{align*}
\nu(\psi^*(h_{11}\textup{d}x+h_{12}\textup{d}y))&\geq\nu(\psi^*(\nu_0x\textup{d}y-\nu_1y\textup{d}x))\\ \nu(\psi^*(h_{21}\textup{d}x+h_{22}\textup{d}y))&\geq\nu(\psi^*(\nu_0y^{\nu_0-1}\textup{d}y-\nu_1x^{\nu_1-1}\textup{d}x)).
\end{align*}
The valuation of $\psi^*(\nu_0x\textup{d}y-\nu_1y\textup{d}x)$ is equal to 
$$(\nu_0-k+1)\nu_1-(k-1)\nu_0$$ and does not belong to $(\Gamma_C+\nu_0)\cup (\Gamma_C+\nu_1)$. Moreover, we get $$\nu(\psi^*(\nu_0y^{\nu_0-1}\textup{d}y-\nu_1x^{\nu_1-1}\textup{d}x))>\nu_0\nu_1.$$
Therefore, we observe that $$\nu(\psi^*(\omega_1^{(2)}))\in(\Gamma_C+\nu_0)\cup(\Gamma_C+\nu_1)$$ for any $h_{21},h_{22}\in\mathbb{C}\{x,y\}$. Thus, denoting $\omega_1=\nu_0x\textup{d}y-\nu_1y\textup{d}x$, we obtain one more minimal generator for $\Lambda_C$ as
$$\lambda_1=\nu(\psi^*(\omega_1))=\gamma+\nu_0=(\nu_0-k+1)\nu_1-(k-1)\nu_0.$$
Beyond $\lambda_1$ the next possible minimal generator for $\Lambda_C$ is obtained considering
\begin{align*}
\omega_2^{(1)}&=\nu_1x^{k-1}\omega_1+\nu_0(\gamma-\nu_1)y^{\nu_0-k}\textup{d}y+h_{11}\textup{d}x+h_{12}\textup{d}y+h_{13}\omega_1\\ 
\omega_2^{(2)}&=y^{k-1}\omega_1+(\gamma-\nu_1)x^{\nu_1-k}\textup{d}x+h_{21}\textup{d}x+h_{22}\textup{d}y+h_{23}\omega_1
\end{align*}
with $h_{ij}\in\mathbb{C}\{x,y\}$ and 
\begin{align*}
\nu(\psi^*(h_{11}\textup{d}x+h_{12}\textup{d}y+h_{13}\omega_1))&\geq\nu(\psi^*(\nu_1x^{k-1}\omega_1+\nu_0(\gamma-\nu_1)y^{\nu_0-k}\textup{d}y))\\ \nu(\psi^*(h_{21}\textup{d}x+h_{22}\textup{d}y+h_{23}\omega_1))&\geq\nu(\psi^*(y^{k-1}\omega_1+(\gamma-\nu_1)x^{\nu_1-k}\textup{d}x)).
\end{align*} 
Notice that $\nu_1x^{k-1}\omega_1+\nu_0(\gamma-\nu_1)y^{\nu_0-k}\textup{d}y$ is precisely the $1$-form $\omega$ given in (\ref{1form}). This implies that $\psi^*(\omega)=0$ ; so $\omega_2^{(1)}$ does not produce any new minimal generator for $\Lambda_C$. 

On the other hand, we remark that any integer $n$ such that $$n\geq \nu(\omega_2^{(2)})>(\nu_1-k+1)\nu_0$$ belongs to $\bigcup_{i=-1}^{1}(\Gamma_C+\lambda_i)$. Indeed, any integer $n$ can be uniquely expressed as $n=\alpha\nu_1+\beta\nu_0$ with $0\leq\alpha<\nu_0$ and $\beta\in\mathbb{Z}$. If $\beta\geq 0$ then $n$ belongs to $(\Gamma_C+\nu_0)\cup (\Gamma_C+\nu_1)$. If $\beta<0$ then the condition $n>(\nu_1-k+1)\nu_0$ implies $\alpha\leq \nu_0-k$ and $k-1+\beta>1$. Consequently, there exists $\delta\in\mathbb{N}$ such that
\begin{align*}
n&=\alpha\nu_1+\beta\nu_0\\ &=\delta\nu_1+(k-1+\beta)\nu_0+(\nu_0-k+1)\nu_1-(k-1)\nu_0\in\Gamma_C+\lambda_1.  
\end{align*}
So, $\omega_2^{(2)}$ does not produce any new minimal generator for $\Lambda_C$ and we conclude that
the minimal generators for $\Lambda_C$ are $$\{\lambda_{-1}=\nu_0,\ \ \lambda_{0}=\nu_1,\ \ \lambda_1=\gamma+\nu_0=(\nu_0-k+1)\nu_1-(k-1)\nu_0\}.$$
 
Computing the integers introduced at (\ref{integers}) we obtain $t_0=\nu_1$ and
$$\begin{array}{lll}
	u_1^{0}=\nu_0+\nu_1 & & u_1^{1}=\nu_0\nu_1 \\
 t_1^{0}=t_1=\nu_0+\nu_1 & & t_1^{1}=\tilde{t}_1=\nu_0\nu_1 \\
	& & \\
	u_2^{0}=(\nu_0-k+1)\nu_1 & & u_2^{1}=(\nu_1-k+1)\nu_0 \\
	t^{0}_2=t_2=k\nu_0+\nu_1 & & t^{1}_2=\tilde{t}_2=k\nu_1+\nu_0.
\end{array}$$

	Therefore, we get $\nu_{D}\left(\omega\right)=k\nu_0+\nu_1=t_{2}$. By Theorem \ref{theo-CanoCorral}, there exists a differential $1$-form $\tilde{\omega}$ with $\nu_{D}\left(\tilde{\omega}\right)=\tilde{t}_{2}=k\nu_1+\nu_0$ such that $\{\omega,\tilde{\omega}\}$ is a Saito basis for $C$.

	Suppose that in the Taylor expansion of $\tilde{\omega}$ there exists a term
	of the form 
	\[
	a_{ij}x^{i}y^{j}\textup{d}x\ \ \textup{ or }\ \ a_{ij}x^{i}y^{j}\textup{d}y
	\]
	such that $a_{ij}\neq 0$ and $\nu\left(x^{i}y^{j}\right)=i+j\leq k-1$. Then, we can see that 
	\[
	\max\{\nu_{D}\left(x^{i}y^{j}\textup{d}x\right ),\nu_{D}\left(x^{i}y^{j}\textup{d}y\right)\}\leq \left(i+j+1\right)\nu_1\leq k\nu_1<k\nu_1+\nu_0=\tilde{t}_{2}
	\]
	which is impossible. Therefore, we get 
	$\nu\left(\tilde{\omega}\right)\geq k$
	and 
	\[
	\mathfrak{s}(C)=\nu(\omega)=k.
	\]
\end{proof}

\subsection{Saito numbers in the topological class $\textup{Top}(C_{N,1,1})$}
Among the curves $C_{N,\nu_0,\nu_1}$, the curves $C_{N,1,1}$ with $N\geq 3$ are the only radial ones, that is, the maximum Saito number is generically realized by a dicritical differential $1-$form \cite{genzmer2020saito}. 

Denoting by $\mathfrak{M}_N$ the number $\max_{C\in\textup{Top}\left(C_{N,1,1}\right)}\mathfrak{s}(C)$, we get 
\[
\mathfrak{M}_N=\left\{ \begin{array}{rl}
0 & N=1\\
1 & N=2,3,4\\
\frac{N-1}{2} & N\geq5\textup{ and \ensuremath{N} odd}\\
\frac{N}{2}-1 & N\geq5\textup{ and \ensuremath{N} even.}
\end{array}\right.
\]

\begin{proposition}
For any $N>1$ and any $k$ in $\left\{ 1,\ldots,\mathfrak{M}_N\right\} $,
there exists $C$ in $\textup{Top}\left(C_{N,1,1}\right)$ such that 
\[
\mathfrak{s}(C)=k.
\]
\end{proposition}

\begin{proof}
If $k=\mathfrak{M}_N$ or if $N=2,3$ or $4$ the statement is obvious. Suppose that $N\geq 5$ and let
$k$ be an integer in $\left\{ 1,\ldots,\mathfrak{M}_N-1\right\} $.
Let us consider the differential $1-$forms
\[
\omega_1=y\dd{x}-x\dd{y},~\omega_2=\omega_1+\dd f
\]
where $f$ is the function 
\[
f=x^{N-2k+2}+y^{N-2k+2}.
\]
The couple $\left\{ \omega_1, \omega_2\right\} $ is a Saito basis for $\{f=0\}$ since it
satisfies the Saito criterion, 
\[
\omega_1\wedge \omega_2=(N-2k+2)f\dd x\wedge\dd y.
\]
Since $N\geq 5$, the multiplicity of $\dd f$ is bigger than $2$ and after one blowing-up, both $\omega_i$'s are dicritical. Let us consider $l_{1}^{(i)}=0,\ldots,l_{k-1}^{(i)}=0$ be
$k-1$ smooth and transversal curves tangent to $\omega_{i}.$ We suppose
moreover that these curves are transversal at a whole and transversal
to $f=0.$ Now writing 
\begin{equation}
\prod_{i=1}^{k-1}l_{i}^{(2)}\omega_{1}\wedge\prod_{i=1}^{k-1}l_{i}^{(1)}\omega_{2}={(N-2k+2)}\prod_{i=1}^{k-1}l_{i}^{(2)}l_{i}^{(1)}f\dd x\wedge \dd y\label{eq:saito}
\end{equation}
yields a Saito relation for the curve $C=\left\{ \prod_{i=1}^{k-1}l_{i}^{(2)}l_{i}^{(1)}f=0\right\} $,
a curve which consists in the union of $N$ smooth and transversal
curves. Thus, $C$ is equisingular to $C_{N,1,1}$ and, following
(\ref{eq:saito}), one has 
\[
\mathfrak{s}\left(C\right)=k.
\]
\end{proof}

\subsection{Saito numbers in the topological class $\textup{Top}(C_{N,\nu_0,\nu_1})$. }

Let $C_{N,\nu_{0},\nu_{1}}$ be the curve given by the equation
\[f_{N,\nu_{0},\nu_{1}}=y^{N\nu_{0}}-x^{N\nu_{1}}=0\]
where $N>0$ and $\nu_{0}\leq\nu_{1}$ are relatively prime. Considering the divisorial valuation defined in (\ref{divisorial-function}) we get the following result.

\begin{lemma}\label{lemma-dicritical}
Let $\omega$ be the $1-$form be defined by 
$\omega=\nu_{1}y\textup{d}x-\nu_{0}x\textup{d}y$.
Suppose that $h=x^{i}y^{j}$. If $\nu_{D}\left(h\right)\leq N\nu_{0}\nu_{1}-1-\nu_{0}-\nu_{1}$
then $\textup{d}f_{N,\nu_0,\nu_1}+h\omega$ is dicritical along the central component of the desingularization process of $C_{N,\nu_0,\nu_1}$.
\end{lemma}

\begin{proof}
Let $u,v$ such that $u\nu_{0}-v\nu_{1}=1$. The local coordinates
in the neighborhood of the central component $D$ can be written
\[
x=x_{D}^{\nu_{0}-v}y_{D}^{\nu_{0}}\quad y=x_{D}^{\nu_{1}-u}y_{D}^{\nu_{1}}.
\]
$y_{D}=0$ being a local equation for $D$ \cite{PaulGen}. Computing the pullback
of $\textup{d}f_{N,\nu_0,\nu_1}+h\omega$ in these coordinates yields 
\begin{align*}
y_{D}^{N\nu_{0}\nu_{1}-1}\left(y_{D}\left(\cdots\right)\textup{d}x_{D}+x_{D}\left(\cdots\right)\textup{d}y_{D}\right)\\
+h\left(x_{D}^{\nu_{0}-v}y_{D}^{\nu_{0}},x_{D}^{\nu_{1}-u}y_{D}^{\nu_{1}}\right)y_{D}^{\nu_{0}+\nu_{1}}x_{D}^{\nu_{0}-v+\nu_{1}-u-1}\textup{d}x_{D}.
\end{align*}
The hypothesis of the lemma ensures that the above $1-$form
can be exactly divided by $y_{D}^{\nu_{D}\left(h\right)+\nu_{0}+\nu_{1}}$
and the $1-$form 
\begin{align*}
y_{D}^{N\nu_{0}\nu_{1}-1-\nu_{0}-\nu_{1}-\nu_{D}\left(h\right)}\left(y_{D}\left(\cdots\right)\textup{d}x_{D}+x_{D}\left(\cdots\right)\textup{d}y_{D}\right)\\
+\frac{h\left(x_{D}^{\nu_{0}-v}y_{D}^{\nu_{0}},x_{D}^{\nu_{1}-u}y_{D}^{\nu_{1}}\right)}{y_{D}^{\nu_{D\left(h\right)}}}x_{D}^{\nu_{0}-v+\nu_{1}-u-1}\textup{d}x_{D}
\end{align*}
is generically transverse to $y_{D}=0.$
\end{proof}
We are now in position to prove Theorem \ref{THEO2} stated in the introduction.
\begin{theorem*}
Let $C_{N,\nu_0,\nu_1}$ be the curve defined by 
\[
y^{N\nu_{0}}-x^{N\nu_{1}}=0
\]
where $N>0$ and $\nu_{0}\leq\nu_{1}$ are relatively prime with $N\nu_0>1$. Then for any {$k\in\{1,\ldots ,\left[\frac{N\nu_{0}}{2}\right]\}$}
there exist $C\in\textup{Top}\left(C_{N,\nu_0,\nu_1}\right)$ such that 
\[
\mathfrak{s}(C)=k.
\]
\end{theorem*}

\begin{proof}
For $N=1$ the result follows from Proposition \ref{irred}. Suppose $N\geq 3$ and the result is true for any curve $C_{N^{\prime},\nu_0,\nu_1}$ with $1\leq N^{\prime}<N$. Let {$k\in\{1,\ldots ,\left[\frac{N\nu_{0}}{2}\right]\}$}. For now,
suppose that $k\geq\nu_{0}+1.$ Let us consider $C^\prime\in\textup{Top}(C_{N-2},\nu_0,\nu_1)$
such that $\mathfrak{s}(C^\prime)=k-\nu_{0}$ and let $\left\{ \omega,\eta\right\} $
be a Saito basis for $C^\prime$. According to the Saito criterion, if $f$
is a reduced equation of $C^\prime$ then 
\[
\omega\wedge\eta=uf\textup{d}x\wedge\textup{d}y
\]
{where $u\in\mathbb{C}\{x,y\}$ is a unit.} Both $1-$forms $\omega$ and $\eta$ can be supposed of multiplicity
$k-\nu_{0}$ and are dicritical along the central component {of the desingularization process of $C'$, that is the same of $C_{N-2}$}. Consider
$f_{\omega}$ and $f_{\eta}$ two reduced equations of analytical
curves invariants for respectively $\omega$ and $\eta$ such that
after desingularization, $f_{\omega}=0$ and $f_{\eta}=0$ are attached
to the common central component, transversal the one to the other and both transversal to $C_{N-2}$. By construction, these two curves
are equisingular to $y^{\nu_{0}}-x^{\nu_{1}}=0.$ Now, we can write
\[
f_{\eta}\omega\wedge f_{\omega}\eta=uf_{\omega}f_{\eta}f\textup{d}x\wedge\textup{d}y.
\]
Since $f_{\eta}\omega$ and $f_{\omega}\eta$ leave both invariant
the curve $ff_{\eta}f_{\omega}=0$, the relation above is the Saito
criterion for $f_{\omega}f_{\eta}f=0.$ Setting $C=\left\{ f_{\omega}f_{\eta}f=0\right\} \in\textup{Top}\left(C_{N,\nu_0,\nu_0}\right),$
we obtain 
\[
\mathfrak{s}(C)=\min\left\{\nu\left(f_{\eta}\omega\right),\nu\left(f_{\omega}\eta\right)\right\}=\nu_0+k-\nu_{0}=k.
\]

Suppose now that, $k\leq\left[\frac{\nu_{0}}{2}\right]$. The initial assumptions ensure that $\nu_0\geq 2$. Applying
the result for $N=1$ yields a curve $C^\prime=\left\{ f=0\right\} \in\textup{Top}\left(C_{1,\nu_0,\nu_1}\right)$
such that $\mathfrak{s}(C^\prime)=k$ and a Saito basis $\left\{ \omega,\eta\right\}$
with $\nu\left(\omega\right)=k\leq\nu\left(\eta\right)$ for $C^{\prime}$ such that the $1-$form $\omega$ is dicritical along the central component {of the desingularization process of $C'$, that is the same of $C_{1,\nu_0,\nu_1}$}.
Choose $N-1$ curves $f_{1}=0,\ldots,$$f_{N-1}=0$ invariant for
$\omega$ attached to the central component {and transversal to} $C^\prime$.
From the Saito criterion {$\omega\wedge\eta=uf\textup{d}x\wedge\textup{d}y$ for $C^\prime$}, we can write 
\[
\omega\wedge f_{1}\cdots f_{N-1}\eta=uf_{1}\cdots f_{N-1}f\textup{d}x\wedge\textup{d}y
\]
which is the Saito criterion for $f_{1}\cdots f_{N-1}f=0$. As a consequence,
setting $C=\left\{ f_{1}\cdots f_{N-1}f=0\right\} ,$ we
get $C\in\textup{Top}\left(C_{N},\nu_0,\nu_1\right)$ and 
\[
\mathfrak{s}(C)=\min\left\{\nu\left(\omega\right),\nu\left(f_{1}\cdots f_{N-1}\eta\right)\right\}=k.
\]
Finally, suppose that {$k\in\{\left[\frac{\nu_{0}}{2}\right]+1,\ldots ,\nu_{0}\}.$}
Consider the curve $C^\prime=\left\{ f=0\right\} $ where $f=y^{2\nu_{0}}-x^{2\nu_{1}}$ and its Saito
basis given by 
\[
\left\{ \omega=\nu_{1}y\textup{d}x-\nu_{0}x\textup{d}y,\textup{d}f\right\} .
\]
For any function $h,$ we can write 
\[
\omega\wedge\left(\textup{d}f+h\omega\right)={2\nu_0\nu_1}f\textup{d}x\wedge\textup{d}y.
\]
We choose $h=x^{k-1}$. In that case, 
\[
\nu_{D}\left(h\right)=\nu_{0}\left(k-1\right)\leq\nu_{0}\left(\nu_{0}-1\right)\leq 2\nu_{1}\nu_{0}-1-\nu_{1}-\nu_{0}.
\]
Following Lemma \ref{lemma-dicritical}, the $1-$form $\textup{d}f+h\omega$ is still dicritical along the central component
$D$ {of the desingularization process of $C'$}. We fix $N-2$ curves $f_{1}=0,\ f_2=0,\ldots ,f_{N-2}=0$ invarant for $\textup{d}f+h\omega$ {attached to the central component $D$ and transversal to $C^{\prime}$}. The Saito criterion leads
to 
\[
f_{1}\cdots f_{N-2}\omega\wedge\left(\textup{d}f+h\omega\right)=2\nu_0\nu_1f_{1}\cdots f_{N-2}f\textup{d}x\wedge\textup{d}y
\]
which is the Saito basis for $C=\left\{ f_{1}\cdots f_{N-2}f=0\right\}$.
Thus 
\begin{align*}
\mathfrak{s}(C) & =\min\left\{\nu\left(f_{1}\cdots f_{N-2}\omega\right),\nu\left(\textup{d}f+h\omega\right)\right\}\\
 & =\min\left\{\left(N-2\right)\nu_{0}+1,2\nu_{0}-1,k\right\}=k
\end{align*}
since $N\geq3.$

It remains to treat the case $N=2.$ Suppose first $\nu_{0}\ge2.$
If $k\leq\left[\frac{\nu_{0}}{2}\right]$ the same argument as before
ensures the property. Suppose that {$k\in\{\left[\frac{\nu_{0}}{2}\right]+1,\ldots ,\nu_{0}\}.$}
Notice that according to \cite{moduligenz}, the property is true for $k=\nu_{0}$. Thus we can suppose $k\leq\nu_{0}-1$. Consider the curve
$C^\prime=\left\{ f=y^{\nu_{0}}-x^{\nu_{1}}=0\right\}$ and its Saito basis given by 
\[
\left\{ \omega=\nu_{1}y\textup{d}x-\nu_{0}x\textup{d}y,\textup{d}f\right\} .
\]
For any function $h,$ we get $
\omega\wedge\left(\textup{d}f+h\omega\right)=\nu_0\nu_1f\textup{d}x\wedge\textup{d}y$. In particular, for $h=x^{k-1}$ we have 
\[
\nu_{D}\left(h\right)=\nu_{0}\left(k-1\right)\leq\nu_{0}\left(\nu_{0}-2\right)\leq\nu_{1}\nu_{0}-1-\nu_{1}-\nu_{0}
\]
which is true under the assumption the assumption that $\nu_{0}\geq2.$
Following Lemma \ref{lemma-dicritical}, the $1-$form $\textup{d}f+h\omega$ is dicritical along the central component
$D$ {of the desingularization process of $C'$}. We choose one curve $f_{1}=0$ invariant for $\textup{d}f+h\omega$
attached to $D$ and transversal to $C^{\prime}$. The Saito criterion leads to 
\[
f_{1}\omega\wedge\left(\textup{d}f+h\omega\right)=\nu_0\nu_1f_{1}f\textup{d}x\wedge\textup{d}y
\]
which $\{f_1\omega, \textup{d}f+h\omega\}$ is the Saito basis for $C=\left\{ f_{1}f=0\right\}$. Expanding the expression of the form $\textup{d}f+h\omega$, we get 
\[
\nu\left(\textup{d}f+h\omega\right)=k
\]
and thus
\begin{align*}
\mathfrak{s}(C) & =\min\left\{\nu\left(f_{1}\omega\right),\nu\left(\textup{d}f+h\omega\right)\right\}\\
 & =\min\left\{\nu_{0}+1,k\right\}=k.
\end{align*}
For $N=2$ and $\nu_{0}=1$, then the set $\{1,\ldots ,\left[\frac{N\nu_{0}}{2}\right]\}$
reduces to integer $1$ and the property is clear regarless the value
of $\nu_{1},$ which concludes the proof of the theorem.
\end{proof}

\bibliographystyle{plain}


\vspace{0.75cm}

\noindent {\sc Yohann Genzmer}\\
	Universit\'e Paul Sabatier \\
    Institut de Math\'ematiques de Toulouse \\
	118 route de Narbonne \\
    F-31062, Toulouse Cedex 9 \\
	France. 
	
\noindent {yohann.genzmer@math.univ-toulouse.fr}
	\vspace{0.5cm}

\noindent {\sc Marcelo Escudeiro Hernandes}\\
	Universidade Estadual de Maring\'a \\
    Departamento de Matem\'atica \\
	Av. Colombo 5790 \\ 
    Maring\'a-PR 87020-900 \\
	Brazil. 
	
\noindent {mehernandes@uem.br}
	\vspace{0.3cm}
\end{document}